\documentclass{birkmult}
\newtheorem{theorem}{Theorem}[section]
\newtheorem{proposition}[theorem]{Proposition}

\newtheorem{defn}[theorem]{Definition}

\newtheorem*{theorem*}{Theorem}
\theoremstyle{definition}

\newtheorem{example}[theorem]{Example}
\newtheorem{remark}[theorem]{Remark}

\newcommand{\R}{\mathbb{R}}
\newcommand{\C}{\mathbb{C}}
\newcommand{\Z}{\mathbb{Z}}
\newcommand{\N}{\mathbb{N}}
\newcommand{\Q}{\mathbb{Q}}
\newcommand{\PP}{\mathbb{P}}

\newcommand{\M}{\mathcal{M}}
\newcommand{\V}{\mathcal{V}}

\renewcommand\epsilon{\varepsilon}

\begin{document}
%
%
%
%
%
%
\title[Computing Linear Matrix Representations]{Computing Linear Matrix Representations  \\
of Helton-Vinnikov Curves}
\author{Daniel Plaumann}

\address{Fachbereich Mathematik und Statistik, Universit\"at Konstanz,~78457 Konstanz, Germany}

\email{Daniel.Plaumann@uni-konstanz.de}

\thanks{Daniel Plaumann was supported by the Alexander-von-Humboldt 
Foundation through a Feodor Lynen postdoctoral fellowship,
hosted at UC Berkeley.
Bernd Sturmfels and Cynthia Vinzant acknowledge support by
 the U.S.~National Science Foundation
(DMS-0757207 and DMS-0968882).
}
\author{Bernd Sturmfels}
\address{Dept.~of Mathematics, University of
  California, Berkeley, CA 94720, USA}
\email{bernd@math.berkeley.edu}

\author{Cynthia Vinzant}
\address{Dept.~of Mathematics, University of
  Michigan, Ann Arbor, MI 48109, USA}
\email{vinzant@umich.edu}

\subjclass[2010]{Primary: 14Q05;  Secondary: 14K25}

\keywords{Plane curves, symmetric determinantal representations,
spectrahedra, linear matrix inequalities, hyperbolic
  polynomials, theta functions}

\date{January 1, 2004}
\dedicatory{Dedicated to Bill Helton on the occasion of his 65th birthday.}

\begin{abstract}
Helton and Vinnikov showed that every
rigidly convex curve in the real plane bounds
a spectrahedron. This leads to the computational problem of
explicitly producing a symmetric (positive definite) linear determinantal representation
for a given curve.
We study three approaches to this problem:
an algebraic approach via solving polynomial equations,
a geometric approach via contact curves,
and an analytic approach via theta functions.
These are explained, compared,
and tested experimentally for
low degree instances.
\end{abstract}

\maketitle

\section{Introduction}

The Helton-Vinnikov Theorem \cite{HV} gives
a geometric characterization of two-dimensional spectrahedra.
They are precisely the subsets of $\R^2$ that are bounded by
rigidly convex algebraic curves, here called \textit{Helton-Vinnikov curves}.  
These curves are cut out by {\em hyperbolic polynomials} in three variables, as discussed in~\cite{LPR}.
This theorem is a refinement of a result from classical algebraic
geometry which states that every homogeneous polynomial
in three variables can be written~as
\begin{equation}
\label{lmr1}
f(x,y,z) \,\, = \,\, {\rm det}(A x + By + Cz)
\end{equation}
where $A, B$ and $C$ are symmetric matrices.  Here the coefficients of
$f$ and the matrix entries are complex numbers.  When the coefficients
of $f$ are real then it is desirable to find $A,B$ and $C$ with real
entries.  The representations relevant for spectrahedra are the
\emph{real definite representations}, which means that the linear span of the
real matrices $A,B$ and $C$ contain a positive definite matrix.  Such
a representation is possible if and only if the corresponding curve
$\{(x:y:z) \in \PP^2_\R : f(x,y,z) = 0 \}$ is {\em rigidly
  convex}. This condition means that the curve has the maximal number
of nested ovals, namely, there are $d/2$ resp.~$(d-1)/2$ nested ovals
when the degree $d$ of $f$ is even resp.~odd.  The innermost oval
bounds a spectrahedron.

Two linear matrix representations
$Ax+By+Cz$ and $A'x+B'y+C'z$  of the same plane curve are said to be
{\em equivalent} if they lie in the same orbit under conjugation, i.e.~if 
there exists an invertible complex matrix $U$ that satisfies
$$ U \cdot (Ax+By+Cz) \cdot U^T \quad = \quad  A'x+B'y+C'z. $$
We call an equivalence class of complex representations \emph{real}
(resp.~\emph{real definite}) if it contains a real (resp.~real
definite) representative. Deciding whether a given complex
representation is equivalent to a real or real definite one is rather
difficult. 

We shall see that the number of equivalence classes of
complex representations (\ref{lmr1}) is finite,
and, for smooth curves, the precise number is known
(Thm.~\ref{equivcount}). Using more general results
of Vinnikov \cite{Vin}, we also derive the number of real and real
definite equivalence classes.
If a Helton-Vinnikov curve is smooth then  the number
of real definite equivalence classes equals
$2^g$, where $g = \binom{d-1}{2}$ is the genus.
 
This paper concerns the computational problem of constructing one
representative from each equivalence class for a given polynomial $f(x,y,z)$.
 As a warm-up example,
consider the following elliptic curve in Weierstrass normal form:
$$ f(x,y,z) \quad = \quad  (x+ay)(x+by)(x+cy) - xz^2 .$$
Here $a ,b,c$ are distinct non-zero reals.
This cubic has
precisely three inequivalent linear symmetric determinantal representations over $\C$,
given by the matrices
  \[\begin{footnotesize}
  \begin{bmatrix}
    x+ay & z\sqrt {{\frac {b}{b-c}}} & z\sqrt {{ \frac {c}{c-b}}}\\
    z\sqrt {{\frac {b}{b-c}}}&x+by&0 \\ z\sqrt {{\frac
        {c}{c-b}}}&0&x+cy
  \end{bmatrix}\! ,
  \begin{bmatrix}
    x+ay&z\sqrt {{\frac {a}{a-c}}}&0 \\ z\sqrt {{\frac
        {a}{a-c}}}&x+by&z\sqrt {{\frac {c}{c- a}}}\\ 0&z\sqrt {{\frac
        {c}{c-a}}}&x+cy
  \end{bmatrix}\! ,
  \begin{bmatrix}
    x+ay&0&z\sqrt {{\frac {a}{a-b}}} \\ 0&x+by&z\sqrt {{\frac
        {b}{b-a}}}\\ z\sqrt {{\frac {a}{a-b}}}&z\sqrt {{\frac
        {b}{b-a}}}&x+cy
  \end{bmatrix}\end{footnotesize}
  \]
All three matrices are non-real if $a,b$ and $c$ have the
same sign,
and otherwise two of the matrices are real. For instance, if
 $a < 0$ and $0 < b < c$ then the first two matrices are real.
 In that case, 
  the cubic is a Helton-Vinnikov curve, and its bounded region,
when  drawn
  in the affine plane $\{x=1\}$, is the spectrahedron
 $$ 
 \left\{\, (y,z) \in \R^2\,\,:\,\,
  \left[ \begin {array}{ccc} 1+ay&z\sqrt {{\frac {a}{a-c}}}&0
\\\noalign{\medskip}z\sqrt {{\frac {a}{a-c}}}&1+by&z\sqrt {{\frac {c}{c-
a}}}\\\noalign{\medskip}0&z\sqrt {{\frac {c}{c-a}}}&1+cy\end {array}
 \right]  \succeq \,0 \,\right\} .$$
 The symbol ``$\succeq$'' means that the matrix is positive semidefinite.
 This spectrahedron  is depicted in Figure \ref{fig:cubic}
 for the parameter values $ a=-1$, $b=1$ and $c=2$.

\begin{figure}
 \includegraphics[width=6.0cm]{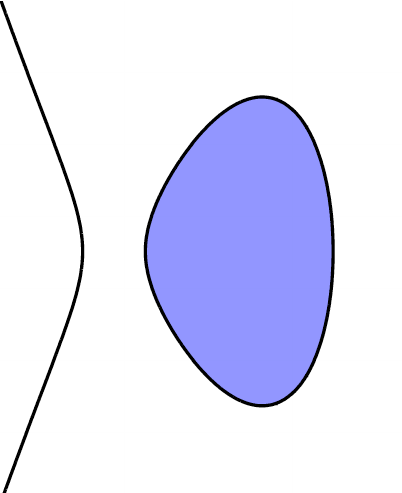} \quad
\vskip -0.2cm
\caption{A cubic Helton-Vinnikov curve and its spectrahedron}
\label{fig:cubic}
\end{figure}

This article is organized as follows. In Section 2 we translate (\ref{lmr1}) into a 
system of polynomial equations in the matrix entries of $A,B,C$, 
we determine the number of solutions (in Theorem \ref{equivcount}),
and we discuss practical aspects of computing these solutions using
both symbolic and numeric software.
Section 3 is devoted to geometric constructions for obtaining
the representation (\ref{lmr1}). Following Dixon \cite{Dix}, these require finding
contact curves of degree $d-1$ for the given curve of degree $d$.

An explicit formula for (\ref{lmr1}) appears in the article of Helton and Vinnikov \cite[Eq.~4.2]{HV}.
That formula requires the numerical evaluation of Abelian integrals and theta functions.
In Section 4, we explain the Helton-Vinnikov formula,
and we report on our computational experience with the implementations 
of  \cite{DeEtAl, DP, DvH}  in the {\tt Maple} package {\tt algcurves}.
In Section 5 we focus on the case of quartic polynomials and relate our
results in \cite{PSV} to the combinatorics of  theta functions.
Smooth quartics have $36$ inequivalent representations (\ref{lmr1}).
In the Helton-Vinnikov case, twelve of these are real, but only 
eight are real definite. One of our findings
is an explicit quartic in $\Q[x,y,z]$ that has all of its $12$ real representations over~$\Q$.

\section{Solving Polynomial Equations}\label{sec:PolynomialEquations}

Our given input is a homogeneous polynomial 
$f(x,y,z)$ of degree $d$, usually over $\Q$.
We assume for simplicity that the corresponding
curve  in the complex projective plane is smooth,
we normalize so that $f(x,0,0) = x^d$, and we further assume that
the factors of the binary form $\,f(x,y,0) = \prod_{i=1}^d (x+ \beta_i y)$ are distinct.
Under these hypotheses, every equivalence class of 
 representations (\ref{lmr1}) contains a representative where
$A$ is the identity matrix and $B$ is the diagonal matrix with entries
$\beta_1<\beta_2<\cdots<\beta_d$. This follows from the linear algebra fact
that any two quadratic forms with distinct eigenvalues can be diagonalized
simultaneously over $\C$.  See Section IX.3 in Greub's text book \cite{Gre} or the proof of Theorem~4.3 
in \cite{PSV}.

After fixing the choices $A = {\rm diag}(1,1,\ldots,1)$ 
and $B = {\rm diag}(\beta_1,\beta_2,\ldots,\beta_d)$ for the first two matrices, we
are left with the problem of finding the $\binom{d+1}{2}$ entries  of the
symmetric matrix $C = (c_{ij})$. By equating the coefficients of all
terms $x^\alpha y^\beta z^\gamma$ with $\gamma \geq 1$ on both sides of (\ref{lmr1}), we
obtain a system of $\binom{d+1}{2}$ polynomial equations in the $\binom{d+1}{2}$
unknowns $c_{ij}$. More precisely, the coefficient
of $x^\alpha y^\beta z^\gamma$ in  (\ref{lmr1})
leads to an equation of degree $\gamma$ in
the $c_{ij}$.
We are thus faced with the problem of solving a square
system of polynomial equations. The expected number
of complex solutions of that system is, according to B\'ezout's Theorem,
\begin{equation}
\label{bezout}
 1^d \cdot 2^{d-1} \cdot 3^{d-2} \cdot 4^{d-3} \cdot \cdots \cdot (d-1)^2  \cdot d . 
 \end{equation}
 This estimate overcounts the number of equivalence classes
 of representations (\ref{lmr1}) because
 we can conjugate the matrix $Ax +By+Cz$ by a diagonal
 matrix whose entries are $+1 $ or $-1$. This conjugation does not change $A$ or $B$
 but it leads to $2^{d-1}$ distinct matrices $C$  all of which are equivalent.
 Hence, we can expect the number
 of inequivalent linear determinantal representations (\ref{lmr1}) to be bounded above by
 \begin{equation}
\label{bezout2}
 3^{d-2} \cdot 4^{d-3} \cdot \cdots \cdot (d-1)^2  \cdot d . 
 \end{equation}
We shall refer to this number as the {\em B\'ezout bound} for our problem.

It is a result in classical algebraic geometry that the number of complex
solutions to our equations is finite, and
the precise number of solutions is in fact known as well.
The following theorem summarizes both what is known for
arbitrary smooth curves over $\C$
and what can be shown for Helton-Vinnikov curves over $\R$:

\begin{theorem} \label{equivcount}
The number of equivalence classes of linear symmetric determinantal representations
{\rm (\ref{lmr1})} of a generic smooth curve of degree $d$ in the projective plane~is 
\begin{equation}
\label{truenumber} 2^{\binom{d-1}{2}-1} \cdot \bigl(\, 2^{\binom{d-1}{2}}+1 \bigr) ,
\end{equation}
unless $d \geq 11$ and $d$ is congruent to $\pm 3$
   modulo $8$, when the number drops by one.
In the case of a Helton-Vinnikov curve, the number of real
equivalence classes of symmetric linear determinantal representations 
{\rm (\ref{lmr1})} is either
$\,
2^{\binom{d-1}{2}-1}(2^{\lceil\frac{d}{2}\rceil-1}+1)\,$
or one less. The number of real definite equivalence classes is precisely $\,
2^{\binom{d-1}{2}}$.
\end{theorem}

\begin{proof}[Sketch of Proof]
The equivalence classes of representations (\ref{lmr1}) correspond to
{\em ineffective even theta characteristics}
 \cite{Bea} on a smooth curve of
genus $g = \binom{d-1}{2}$. The number of even theta characteristics
is $2^{g-1} ( 2^g+1)$, and all even theta characteristics are ineffective
for $d \leq 5$ and $d \equiv 0,1,2,4,6,7$ mod $8$. In all other cases
there is precisely one effective even theta characteristic, provided the
curve is generic. This was shown by Meyer-Brandis in his 1998 diploma thesis \cite{MB},
and it refines results known classically in algebraic
geometry~\cite[Chapters 4-5]{Dol}.
The count of real and real definite representations will be proved at
the end of Section~\ref{sec:theta}.
\end{proof}

The following table lists the numbers in (\ref{bezout2}) and (\ref{truenumber}) for small values of~$d$:
$$
\begin{matrix}  {\rm degree} \,\, d && 2 & 3 & 4 & 5 & 6 & 7 \\
\hline
{\rm genus} \,\, g  && \,\,0\,\, & \,\,1\,\, & \,\,3 \,\,\, & \,\, 6 \,\, & 10 & 15 \\
\text{B\'ezout}\,{\rm bound} && 1 & 3 & 36 & 2160 & 777600 & 1959552000 \\
{\rm True} \, {\rm number}  && 1 & 3 & 36 & 2080 & 524800  & 536887296 \\
\end{matrix} 
$$
This table shows that computing all solutions to our equations 
is a challenge when $d \geq 6$.
Below we shall discuss some computer
experiments we conducted for $d \leq 5$.

\bigskip

As before, we fix $A$ to be the $d {\times} d$ identity matrix, denoted ${\rm Id}_d$,
and we fix $B$ to be the diagonal matrix
with entries $\beta_1 < \cdots < \beta_d$. We also fix the diagonal entries of $C$
since these are determined by solving the $d$  linear equations that arise
by comparing the coefficient  of any of the $d$ monomials $ x^i y^{d-i-1} z$
in (\ref{lmr1}). They are expressed in terms of $f$ and the $\beta_i$ by the following explicit formula:
\begin{equation}
\label{diagC} \quad
 c_{ii} \,\,\, = \,\,\, \beta_i \cdot \frac{ \frac{\partial f}{\partial z}(-\beta_i,1,0)}{
 \frac{\partial f}{\partial y}(-\beta_i,1,0)} \qquad \quad \hbox{for} \,\, i =1,2,\ldots,d.
 \end{equation}

We are thus left with a system of $\binom{d}{2}$ equations in the
$\binom{d}{2}$ off-diagonal unknowns $c_{ij}$. In order to remove the
extraneous factor of $2^{d-1}$ in the B\'ezout bound (\ref{bezout}) coming from 
sign changes on the rows and columns of $C$, we can
perform a multiplicative change of coordinates as follows:
$\,x_{1j} = c_{1j}^2  \,$ for $j = 1 , 2, \ldots, d \,$ and
$\,x_{ij}  = c_{1i} c_{1j} c_{ij}\,$ for $2 \leq i < j \leq d$.
This translates our system of polynomial equations in
the $c_{ij}$ into a system of Laurent polynomial equations in
the $x_{ij}$, and each solution to the latter encodes an 
equivalence class of $2^{d-1}$ solutions to the former.

\begin{example}
Let $d=4$. We shall illustrate the two distinct formulations of the
system of equations to be solved.
We fix a quartic Helton-Vinnikov polynomial
$$ 
f(x,y,z) \;= \; {\rm det}
\begin{footnotesize}
  \begin{bmatrix}
    x + \beta_1 y + \gamma_{11} z & \gamma_{12} z & \gamma_{13}  z & \gamma_{14} z \\
    \gamma_{12} z & x + \beta_2 y + \gamma_{22} z & \gamma_{23} z & \gamma_{24} z \\
    \gamma_{13} z & \gamma_{23} z & x + \beta_3 y + \gamma_{33} z & \gamma_{34} z \\
    \gamma_{14} z & \gamma_{24} z & \gamma_{34} z & x + \beta_4 y +
    \gamma_{44} z
  \end{bmatrix}
\end{footnotesize}
$$
where $\beta_i$ and $\gamma_{jk}$ are rational numbers.
From the quartic $f(x,y,z)$ alone we can recover the $\beta_i$ and 
the diagonal entries $\gamma_{jj}$ as described above. 
 Our aim is now to compute {\bf all}
points $(c_{12}, c_{13}, c_{14}, c_{23}, c_{24}, c_{34}) \in \C^6$ that satisfy the identity
$$ {\rm det}
\begin{footnotesize}
  \begin{bmatrix}
    x + \beta_1 y + \gamma_{11} z & c_{12} z & c_{13} z & c_{14} z \\
    c_{12} z & x + \beta_2 y + \gamma_{22} z & c_{23} z & c_{24} z \\
    c_{13} z & c_{23} z & x + \beta_3 y + \gamma_{33} z & c_{34} z \\
    c_{14} z & c_{24} z & c_{34} z & x + \beta_4 y + \gamma_{44} z
  \end{bmatrix}
\end{footnotesize}
\; = \; f(x,y,z). $$
The coefficient of $z^4$ gives one equation of degree $4$ in the six unknowns $c_{ij}$,
the coefficients of $xz^3$ and $yz^3$ give two cubic equations,
and the coefficient of $x^2 z^2, xy z^2$ and $y^2z^2$ give three quadratic equations
in the $c_{ij}$.
The number of solutions in $\C^6$ to this system of equations is equal to 
$\,2^3 3^2 4 =  288$.
These solutions can be found using symbolic software,
such as  {\tt Singular} \cite{singular}. However, the above formulation 
has the disadvantage that each equivalence class of solutions
appears eight times.

We note that,  for generic choices of  $\beta_i, \gamma_{jk}$, all solutions 
lie in the torus $(\C^*)^6$ where $\C^* = \C \backslash \{0\}$, 
and we shall now assume that this is the case.
Then the
 $8$-fold redundancy can be removed by working
 with the following invariant coordinates:
 \begin{align*}
  &x_{12} = c_{12}^2\,,\,\, x_{13} = c_{13}^2\,,\,\, x_{14} =
   c_{14}^2\,,\,\,\\ 
   &x_{23} = c_{12} c_{13} c_{23}\,,\,\,
   x_{24} =
   c_{12} c_{14} c_{24}\, , \,\, x_{34} = c_{13} c_{14} c_{34} .
 \end{align*}
We rewrite our six equations in these coordinates by  performing the substitution:
\begin{align*}
  &c_{12} = x_{12}^{1/2}\,,\,\, c_{13} = x_{13}^{1/2}\,,\,\, c_{14} =
  x_{14}^{1/2}\,,\,\, \\
  &c_{23} = \frac{x_{23}}{x_{12}^{1/2}
    x_{13}^{1/2}}\,,\,\, c_{24} = \frac{x_{24}}{x_{12}^{1/2}
    x_{14}^{1/2}}\,,\,\, c_{34} = \frac{x_{34}}{x_{13}^{1/2}
    x_{14}^{1/2}}.
\end{align*}
This gives six Laurent polynomial equations  in
six unknowns $x_{12}$, $x_{13}$, $x_{14}$, $x_{23}$, $x_{24}$, $x_{34}$.
They have precisely $36$ solutions in $(\C^*)^6$,
one for each equivalence class. \qed
\end{example}

While the solution of the above equations using symbolic
Gr\"obner-based software is easy for $d=4$, we found that
this is no longer the case for $d \geq 5$. For $d=5$, it was
necessary to employ tools from numerical algebraic
geometry, and we found that {\tt Bertini} \cite{bertini} works well for 
our purpose.
The computation reported below
is due to Charles Chen, an undergraduate student at UC Berkeley.
This was part of Chen's term project in convex algebraic geometry
 during Fall 2010.

\smallskip

For a concrete example,
let us consider the following polynomial which defines a smooth Helton-Vinnikov curve of degree $d=5$:
    \begin{footnotesize}
  \begin{align*}
    f(x,y,z) \, = \,&\,\,\,x^5+3 x^4 y-2 x^4 z-5 x^3 y^2-12 x^3 z^2-15 x^2
    y^3+10 x^2 y^2 z-28 x^2 y z^2+ 14 x^2 z^3 + \\
    &4 x y^4 -6 x y^2 z^2-12 x y z^3+26 x z^4+12 y^5-8 y^4 z-32 y^3
    z^2+ 16 y^2 z^3{+}48 y z^4{-}24 z^5.
  \end{align*}
\end{footnotesize}
The symmetric linear determinantal representation we seek has the form
$$ 
\begin{footnotesize}
  \begin{bmatrix}
    x+y &             0&              0 &             0&              0  \\
    0&        x+2y &            0 &             0 &             0     \\
    0&              0&        x-y &             0&              0   \\
    0&              0&             0&      x-2y  &             0   \\
    0 & 0& 0& 0& x+3y-2z
  \end{bmatrix}
\end{footnotesize}
\,\, + \,\,\, C \cdot z  ,$$
      where $C = (c_{ij})$ is an unknown symmetric $5 {\times} 5$-matrix
      with zeros on the diagonal. This leads to a system of
      $10$ polynomial equations in    the  $10$ unknowns $c_{ij}$,
      namely, $4$ quadrics, $3$ cubics, $2$ quartics and one quintic.
      The number of complex solutions equals  $16 \cdot 2080 = 33280$,
      which is less than the B\'ezout bound of $ 2^4 \cdot 3^3 \cdot 4^2 \cdot 5 = 16 \cdot 2160 = 34560$.
      One of the $33280$ solutions is the following integer matrix,
      which we had used to construct $f(x,y,z)$ in the first place:
$$ C \quad = \quad       
\begin{bmatrix}
 \,0 &  2&  1&  0&  0 \\
\, 2&  0&  0&  0&  1 \\
\, 1&  0&  0&  2&  1 \\
\, 0&  0&  2&  0& -1 \\
\, 0&  1&  1& -1&  0
\end{bmatrix}
$$
Of course, the other $15$ matrices in the same equivalence class
have the same friendly integer entries. The other $16 \cdot 2079 = 33264$ complex solutions
were found numerically using 
the software {\tt Bertini} \cite{bertini}. Of these,
 $16\cdot 63$ are real.
Chen's code, based on {\tt Bertini}, outputs one representative per class.
One of the real solutions~is
$$ C \,\, \approx \,\,
\begin{footnotesize}
  \begin{bmatrix}
    0&  1.8771213868&  0.1333876113&  0.3369345269&  0.2151885297 \\
    1.8771213868&           0&  1.3262201851&  0.1725327846&  1.0570303927 \\
    0.1333876113&  1.3262201851&           0& -2.0093203944& -0.8767796987 \\
    0.3369345269&  0.1725327846& -2.0093203944&           0& -0.7659896773 \\
    0.2151885297& 1.0570303927& -0.8767796987& -0.7659896773& 0
  \end{bmatrix}.
\end{footnotesize}
$$

\section{Constructing Contact Curves}

A result in classical algebraic geometry states that the equivalence classes
of symmetric linear determinantal representations of a 
plane curve of degree $d$ are in one-to-one 
correspondence with certain systems of contact curves of degree $d-1$.
Following \cite[Prop. 4.1.6]{Dol}, we now state this in precise terms.
Suppose that our given polynomial is $f = {\rm det}(M)$ where $M=(\ell_{ij})$
is a symmetric $d {\times} d$-matrix of linear forms in $ x,y,z$.  We can then form the $d\times d$ adjoint matrix, ${\rm adj}(M)$, 
whose entry $m_{ij}$ is the $(i,j)$th $(d-1)$-minor of $M$ multiplied by $(-1)^{i+j}$. For any vector of parameters $u = (u_1,u_2,\ldots,u_d)^T$, 
we consider the degree $d-1$ polynomial
\begin{equation}\label{eq:contact} g_{u}(x,y,z) \quad = \quad u^T{\rm adj}(M) u\quad= \quad \begin{bmatrix} u_1\\u_2\\ \vdots \\ u_d \end{bmatrix}^T \begin{bmatrix}
m_{11} & m_{12} & \cdots & m_{1d} \\
m_{12} & m_{22} & \cdots & m_{2d} \\
  \vdots & \vdots & \ddots & \vdots  \\
m_{1d} & m_{2d} & \cdots & m_{dd} 
 \end{bmatrix} \begin{bmatrix} u_1\\u_2\\ \vdots \\ u_d \end{bmatrix}.\end{equation}
The curve $\V(g_u)$ has degree $d-1$, and it is a 
{\em contact curve}, which means that all intersection points
of $\V(f)$ and $\V(g_u)$ have even multiplicity, generically multiplicity $2$.
  To see this,
we use \cite[Lemma 4.1.7]{Dol}, which states that, for any $u, v \in \C^d$,
 \[g_u(x,y,z) \cdot g_v(x,y,z)\; - \;(u^T{\rm adj}(M) v)^2 \;\in\; \langle \,f\, \rangle
 \quad {\rm in} \,\,\,\C[x,y,z] .\]
In particular, for $u=e_i$, $v=e_j$, this shows that both $\V(m_{ii})$ and $\V(m_{jj})$ are contact curves,
 and $\V(m_{ij})$ meets $\V(f)$ in their $d(d-1)$ contact points.

We say that two contact curves $\V(g_1)$ and $\V(g_2)$ of degree $r$ lie in the same \textit{system} 
if there exists another curve $\V(h)$ of degree $r$ that meets $\V(f)$ precisely in the $r\cdot d$ points
 $\;\V(f,g_1) \cup \V(f,g_2)$.  
A system of 
contact curves is called \textit{syzygetic} if it contains a polynomial of the form $\ell^2 g$, where $\ell$ is linear and  $g$ is a contact curve of degree $r-2$,
 and \textit{azygetic} otherwise.  
 A contact curve of $\V(f)$ is called syzygetic, resp. azygetic,
if it lies in a system that is syzygetic, resp. azygetic.

Dixon \cite{Dix} proved that the contact curves $\V(g_u)$ are azygetic and 
all azygetic contact curves of degree $d-1$ appear as $g_u$ for some determinantal 
representation $f=\det(M)$. In particular, he gives a method of constructing a determinantal
 representation $M$ for $f$ starting from one azygetic contact curve 
of degree $d-1$. 

The input to Dixon's algorithm is an azygetic contact curve $g$ of degree $d-1$
of the given curve $f$ of degree $d$.
Given the two polynomials $f$ and $g$,  the algorithm 
constructs the matrix $\widehat{M}={\rm adj}(M)$ 
in~(\ref{eq:contact}). It proceeds as follows.
Since $\V(g)$ meets $\V(f)$ in $d(d-1)/2$ points, the vector space of polynomials of degree 
$d-1$ vanishing at these points (without multiplicity) has dimension $d$. Let $m_{11}=g$ and extend $m_{11}$
to a basis $\{m_{11}, m_{12}, \hdots, m_{1d}\}$ of this vector space. 
For $i,j \in \{2,3,\hdots, d\}$, the polynomial $m_{1i} m_{1j}$ vanishes to order two on $\V(f,m_{11})$,
so it lies in the ideal $\langle m_{11}, f \rangle$.
Using the Extended Buchberger Algorithm, one finds a 
degree $d-1$ polynomial $m_{ij}$ such that
$\,m_{1i} m_{1j} - m_{11} m_{ij} \in \langle f \rangle $.
The $d {\times} d$-matrix $\widehat{M}= (m_{ij})$ has rank $1$ modulo $\langle f \rangle$, therefore
its $2 {\times} 2 $-minors are multiples of $f$. 
This implies that the adjoint matrix
${\rm adj}(\widehat{M}) = {\rm det}(\widehat{M}) \cdot \widehat{M}^{-1}$ has the form  $\,\lambda f^{d-2} \cdot M \,$
where $\lambda  \in \C \backslash \{0\}$ and
$M$ is a symmetric matrix of linear forms with ${\rm det}(M) = f$.
One could run through this
construction starting from a  syzygetic contact curve,  
but the resulting matrix $\widehat{M}$ would have determinant zero. 
 
The main challenge with Dixon's algorithm is to construct
its input polynomial $g$. Suitable contact curves are not easy to find.
A symbolic implementation of the algorithm may involve
large field extensions, and we found it  equally difficult to implement numerically.
For further discussions see \cite[\S 2.2]{MB}
and \cite[\S 2]{PSV}.
 
\begin{figure}
 \includegraphics[width=6cm]{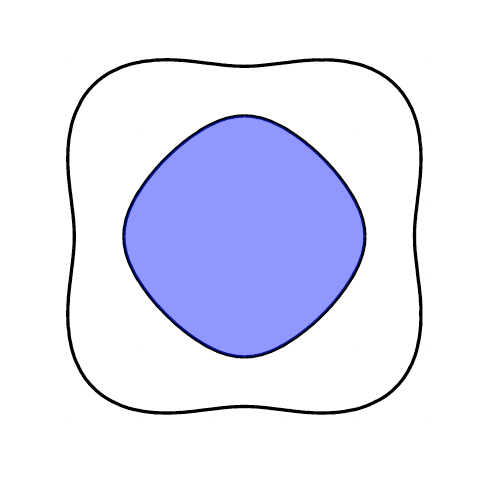} \quad
 \includegraphics[width=6cm]{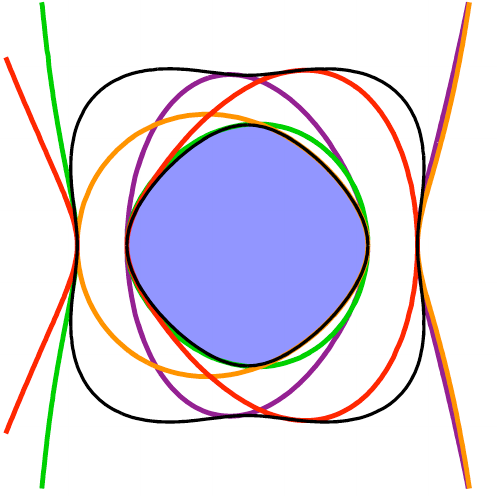} 
\vskip -0.3cm
\caption{A quartic Helton-Vinnikov curve and four contact cubics}
\label{fig:contact}
\end{figure}

\begin{remark} Starting from a \textit{real} azygetic contact curve $g$,
 one can use Dixon's method to produce a \textit{real} 
determinantal representation of $f$.
A determinantal representation $M$ is equivalent to its conjugate $\overline{M}$ if and only if the system of contact curves $\{u^T{\rm adj}(M)u:u\in \C^4\}\subset \C[x,y,z]_3$ is real, i.e.~invariant under conjugation.  The representation $M$ is equivalent to a real matrix if and only if this system contains a real contact curve.  By \cite[Prop~2.2]{GroHa}, if the curve $\V(f)$ has real points, then these two notions of reality agree. However, this approach does not easily reveal whether an 
equivalence class contains a real definite representative.
\end{remark}

\begin{example}
\label{ex:vinnikov}
The following Helton-Vinnikov quartic was studied in \cite[Ex.~4.1]{PSV}:
 $$                                                                              
f(x,y,z) \quad = \quad  2x^4 + y^4 + z^4-3 x^2 y^2-3 x^2 z^2+y^2 z^2.           
$$
It is shown on the left in Figure \ref{fig:contact}. It has
a symmetric determinantal representation
\begin{equation}
\label{ex:hv}
f(x,y,z) \quad = \quad
{\rm det} \begin{bmatrix}
                      u x + y     &  0  &      a z &      b z \\
                          0   &     u  x  - y &       cz &       d z \\
                        a z &     c  z  &   x + y  &    0  \\
			b z  &    d  z  &     0   &   x - y
 \end{bmatrix}, 
 \end{equation}
 where $ a = -0.5746...$, 
 $b =  1.0349...$,
$ c =  0.6997...$,
$d =  0.4800...$
and $u =  \sqrt{2}$ 
are the coordinates of a real zero of 
the following maximal ideal in $\mathbb{Q}[a,b,c,d,u]$:
$$ \begin{matrix} \bigl\langle
u^2-2,
256 d^8 - 384 d^6 u{+}256 d^6{-}  384 d^4 u{+}672 d^4{-}336 d^2 u{+}448 d^2{-}84 u\
{+}121,  \\
\,\,\,\,\,\,\,
23 c + 7584 d^7 u{+}10688 d^7{-}5872 d^5 u{-}8384 d^5{+}1806 d^3 u{+}2452 d^3{-}181 d u{-}307 d, \\
\,\,\,\,\,
23 b + 5760 d^7 u{+}8192 d^7{-}4688 d^5 u{-}6512 d^5{+}1452 d^3 u{+}2200 d^3{-}212 d u{-}232 d, \\
\,\,\,
23 a - 1440 d^7 u{-}2048 d^7{+}1632 d^5 u{+}2272 d^5{-}570 d^3 u{-}872 d^3{+}99 d u{+}81d \,
\bigr\rangle.
\end{matrix}
$$
The principal $3 {\times} 3$-minors of the $4 {\times} 4$-matrix in
(\ref{ex:hv}) are Helton-Vinnikov polynomials of degree $3$.
They are the four contact cubics shown on the right in Figure~\ref{fig:contact}. \qed
\end{example}

In summary,  Dixon's method furnishes an explicit bijection between equivalence
classes of symmetric determinantal representations (\ref{lmr1}) of a fixed curve $\V(f)$ of degree
$d$ and azygetic systems of contact curves of $\V(f)$ of degree $d-1$. \smallskip

For $d=3$, there is another geometric approach to finding representations (\ref{lmr1}). 
We learned this from Didier Henrion who attributes it to Fr\'ed\'eric Han.
Suppose we are given a general homogeneous
cubic $ f(x,y,z)$.  We first compute the Hessian
$${\rm Hes}(f) \quad = \quad
{\rm det} \begin{bmatrix}
\partial^2{f} /\partial{x}^2 & \partial^2{f} /\partial{x}\partial{y} &  \partial^2{f} /\partial{x}\partial{z} \\
 \partial^2{f} /\partial{x}\partial{y} & \partial^2{f} /\partial{y}^2 & \partial^2{f} /\partial{y}\partial{z} \\
 \partial^2{f} /\partial{x}\partial{z} &  \partial^2{f} /\partial{y}\partial{z} & \partial^2{f} /\partial{z}^2 
\end{bmatrix}.
$$
This is also a cubic polynomial, and hence so is the linear combination
$ \,t\cdot f + {\rm Hes}(f) $, where $t$ is a parameter. 
We now take the Hessian of that new cubic, with the aim of recovering~$f$.
It turns out that we can do this by solving
a cubic equation in $t$.

\begin{proposition}
There exist precisely three points $(s,t) \in \C^2$ such that
\begin{equation}
\label{HesEq}
s \cdot f \quad = \quad  {\rm Hes} \bigl( \,t \cdot f +  {\rm Hes}(f) \,\bigr) .
 \end{equation}
The resulting three symmetric determinantal representations of $f$
are inequivalent.
\end{proposition}

\begin{proof}
The statement is invariant under linear changes of coordinates in $\PP^2$,
so, by \cite[Lemma 1]{AD}, we may assume that  the given
cubic is in {\em Hesse normal form}:
$$ f(x,y,z) \quad = \quad x^3+y^3+z^3 - m xyz . $$
In that case, the result follows from the discussion in \cite[page 139]{Hul}.
Alternatively, we can solve the equations obtained by comparing coefficients
in (\ref{HesEq}). This leads~to 
$$\begin{footnotesize}
\begin{matrix}
  t^3\,-\,(12m^4+2592m)t \,-16m^6+8640m^3+93312 \,\,= \,\, 0 \qquad
  \quad \hbox{and}
  \qquad \qquad \qquad \qquad \qquad \\ \quad
  s \,=\,(12m^4+2592m) t^2 \,+\,(48m^6-25920m^3-279936)t \,+\,
  48m^8+20736m^5+2239488m^2.
\end{matrix}\end{footnotesize}
$$
This has three solutions $(s,t)\in \C^3$. The resulting representations (\ref{lmr1}) are
inequivalent because the Hessian normal form of a 
${\rm PGL}(3,\C)$-orbit of cubics is unique.
\end{proof}

\section{Evaluating Theta Functions}\label{sec:theta}

The proof of the Helton-Vinnikov Theorem
relies on a formula, stated in \cite[Eq.~4.2]{HV}, 
that gives a positive definite determinantal
representation of a Helton-Vinnikov curve in terms of theta functions
and the period matrix of the curve.  Our aim in this section is to 
explain this formula and to report on computational experiments with it.
Numerical algorithms for computing theta functions, period
matrices and Abelian integrals have become available in recent years
through work of Bobenko, Deconinck, Heil, van Hoeij, Patterson,
Schmies, and others \cite{DeEtAl, DP, DvH}.
There exists an implementation in \texttt{Maple},
and we used that for our computations. Our
  \texttt{Maple} worksheet that evaluates the Helton-Vinnikov formula
can be found~at
\begin{equation}
\label{eq:url}
  \hbox{\tt www.math.uni-konstanz.de/$\sim$plaumann/theta.html}  
  \end{equation}

Before stating the Helton-Vinnikov formula, we review the basics on
theta functions. Our emphasis will be on clearly defining the
ingredients of the formula rather than explaining the underlying theory.
For general background see \cite{Mum}.  Fix $g \in \N$ and
let $\mathcal{H}^g$ be the \emph{Siegel upper half-space}, which
consists of all complex, symmetric $g\times g$-matrices whose
imaginary part is positive definite. The \emph{Riemann theta function}
is the holomorphic function on $\C^g\times\mathcal{H}^g$ defined by
the exponential~series
\[
\theta({\bf u},\Omega)
\quad = \quad
\sum_{m\in {\Z}^g} \exp\bigl(\pi i(m^T\Omega m+2m^T {\bf u})\bigr),
\]
where $i = \sqrt{-1}$,
$\,{\bf u}=( u_1,\ldots, u_g)\in\C^g\,$ and $\,\Omega\in\mathcal{H}^g$. We
will only need to consider $\theta({\bf u},\Omega)$ as a function in
${\bf u} $, for a fixed matrix $\Omega$, so we may drop $\Omega$ from the
notation. In other words, we define $\theta\colon\C^g\to\C$ by $\theta({\bf u})=\theta({\bf u},\Omega)$.
The theta function is
quasi-periodic with respect to the lattice
$\Z^g+\Omega\Z^g\subset\C^g$, which means 
\[
\theta( {\bf u}+{\bf m}+\Omega {\bf n})\,\,\, = \,\,\,
\exp\bigl(\pi i (-2 {\bf n}^T {\bf u} - {\bf n}^T \Omega {\bf n})\bigr)\cdot\theta( {\bf u}),
\quad\hbox{for all ${\bf m},{\bf n}\in\Z^g$.}
\]
 A \emph{theta characteristic} is
 a vector $\epsilon=  {\bf a}+\Omega {\bf b}
 \in\C^g$ with ${\bf a},{\bf b}\in\{0,\frac{1}{2}\}^g$.  The function
\[
\theta[\epsilon]({\bf u} )\,\,\, = \,\,\,
\exp\bigl(\pi i({\bf a}^T\Omega
{\bf a}+2 {\bf a}^T({\bf u} +{\bf b}))\bigr)\cdot\theta({\bf u}+\Omega {\bf a}+{\bf b})
\]
is the \emph{theta function with characteristic $\epsilon$}. There are
$2^{2g}$ different theta characteristics, indexed by ordered pairs
$(2{\bf a},2 {\bf b})$ of binary vectors in $\{0,1\}^g$.
We also use the notation  $ \,\theta\biggl[
 {\begin{matrix} 2 {\bf a}\\ 2 {\bf b} \end{matrix}} \biggr]( {\bf u})\, $
 for the function $\,\theta[\epsilon]({\bf u})$.
  For $\epsilon = 0$ we simply recover $\theta({\bf u})$.

Let $f\in\C[x,y,z]_d$ be a homogeneous polynomial of degree
$d$. Assume that the projective curve $X=\mathcal{V}_\C(f)$ is smooth
and thus a compact Riemann surface of genus $g=\frac12(d-1)(d-2)$. Let $(\omega_1',\dots,\omega'_g)$ be a basis of the
$g$-dimensional complex vector space of holomorphic $1$-forms on $X$,
and let $\alpha_1,\dots,\alpha_g,\beta_1,\dots,\beta_g$ be closed
$1$-cycles on $X$ that form a {\em symplectic basis} of
$H_1(X,\Z)\cong\Z^g\times\Z^g$. This means that the intersection
numbers of these cycles on $X$ satisfy
$(\alpha_j \cdot \alpha_k)=(\beta_j \cdot \beta_k)=0$, and
$(\alpha_j \cdot \beta_k)=\delta_{jk}$ for all $j,k=1,\dots,g$. The
\emph{period matrix} of the curve $X$ with respect to these bases is the complex
$g\times 2g$-matrix $\bigl(\Omega_1 \,|\, \Omega_2 \bigr )$ whose entries are
\[ (\Omega_1)_{jk\,}=\int_{\alpha_k} \! \omega'_j \quad \hbox{and} \quad
(\Omega_2)_{jk}\,=\int_{\beta_k} \! \omega'_j,\quad 
\hbox{for} \,\,\, j,k=1,\ldots,g. \]
The $g {\times} g$-matrices $\Omega_1$ and $\Omega_2$ are invertible. Performing the coordinate
change
$$(\omega_1,\dots,\omega_g) \,\,\, = \,\,\, (\omega'_1,\dots,\omega'_g) \cdot (\Omega_1^{-1})^T $$
leads to a basis in which the period matrix is of the form
$\bigl(\,{\rm Id}_g\, |\, \Omega_1^{-1}\Omega_2 \,\bigr)$. The basis
$\omega=(\omega_1,\dots,\omega_g)$ is called a \emph{normalized basis
  of differentials} and depends uniquely on the symplectic homology
basis. The $g\times g$-matrix $\Omega=\Omega_1^{-1}\Omega_2$ is
symmetric and lies in the Siegel upper half-space $\mathcal{H}^g$.  It
is called the \emph{Riemann period matrix} of the polynomial
$f(x,y,z)$ with respect to the homology basis
$(\alpha_1,\dots,\alpha_g,\beta_1,\dots,\beta_g)$.

With the given polynomial $f$ we have now associated a system
$\big\{\theta[\epsilon](\,\cdot\,,\Omega) \bigr\}$ of $2^{2g}$ theta
functions with characteristics.  A theta characteristic $\epsilon={\bf
  a}+\Omega {\bf b}$ is called \emph{even} (resp.~\emph{odd}) if the
scalar product $(2 {\bf a})^T(2{\bf b})$ of its binary vector labels
is an even (resp.~odd) integer.  This is equivalent to
$\theta[\epsilon]$ being an even (resp.~odd) function in ${\bf u}$.
In symbols, we have $\,\theta[\epsilon](-{\bf u})=(-1)^{4 {\bf a}^T
  {\bf b}}\theta[\epsilon]({\bf u})$.  Changing symplectic
bases of $H_1(X,\Z)$ corresponds to the right-action of the symplectic
group ${\rm Sp}_{2g}(\Z)$ on the period matrix $\bigl(\Omega_1
\,|\,\Omega_2 \bigr)$. This action will permute the theta
characteristics. In particular, there is no distinguished even theta
characteristic $0$.

Finally, we define the \emph{Abel-Jacobi map} by
$\phi(P)=(\int_{P_0}^P\omega_1,\dots,\int_{P_0}^P\omega_g)^T$ for
$P\in X$, where $P_0\in X$ is any fixed base point. This is a
holomorphic map, but it is well-defined only up to the period lattice
$\Lambda=\Z^g+\Omega\Z^g\subset\C^g$. In other words, the Abel-Jacobi
map is a holomorphic map $\phi\colon X\to{\rm Jac}(X)=\C^g/\Lambda$.

\smallskip

We are now ready to state the formula for (\ref{lmr1})
in terms of theta functions.

\begin{theorem}[Helton-Vinnikov \cite{HV}] \label{hvthm} Let $f \in \R[x,y,z]_d$ with $f(1,0,0)= 1$ 
and let $X$ denote $\mathcal{V}_\C(f)\subset \PP^2$.
We make the following two assumptions:
\begin{enumerate}
\item The curve $X$ is a non-rational Helton-Vinnikov curve with
  the point $(1:0:0)$ inside its innermost oval. The latter means that,
  for all $v \in \R^3 \backslash \{0\}$, the univariate  polynomial 
  $f(v+t\cdot(1,0,0))\in\R[t]$ has only real zeros.
    \item The $d$ real intersection points of $X$ with
    the line $\{z=0\}$ are distinct non-singular points 
    $Q_1,\ldots,Q_d$, with coordinates
        $Q_i = (-\beta_j:1: 0)$
    where $\beta_j \not=0$.
\end{enumerate}
Then $\,f(x,y,z)\, = \, {\rm det}({\rm Id}_d x + By + Cz)\,$
where $B = {\rm diag}(\beta_1,\ldots,\beta_d)$ and $C$ is real symmetric
with diagonal entries $c_{jj}$ as in  (\ref{diagC}). The
 off-diagonal entries of $C$ are 
\begin{equation}
\label{amazingformula}
 c_{jk} =
\frac{\beta_k-\beta_j}{ \theta[\delta](0)}
\cdot \frac{\theta[\delta]\bigl(\phi(Q_k) - \phi(Q_j)\bigr)}
{\theta[\epsilon]\bigl(\phi(Q_k) - \phi(Q_j)\bigr)} 
\sqrt{\frac{ \,\omega \cdot \nabla\theta[\epsilon](0) }
{-d(z/y)}(Q_j)} 
\sqrt{\frac{\,\omega \cdot \nabla\theta[\epsilon](0)}
{-d (z/y) }(Q_k)}. 
\end{equation}
Here $\epsilon$ is an arbitrary odd theta characteristic
and $\delta$ is a suitable even theta characteristic with $\theta[\delta](0)\neq 0$.
The theta functions are taken with respect to a normalized basis of
differentials $\omega=(\omega_1,\dots,\omega_g)$, and $\phi\colon
X\to{\rm Jac}(X)$ is the Abel-Jacobi map.
\end{theorem}

The remarkable expression for the constants $c_{jk}$ in 
(\ref{amazingformula}) does not depend on the choice of the odd
characteristic $\epsilon$.  If the curve $X$ is smooth, then all
equivalence classes of symmetric determinantal representations are
obtained when $\delta$ runs through all non-vanishing even
theta characteristics.
The proof of Theorem \ref{hvthm}
given in \cite{HV} is only an outline. 
It relies heavily on earlier results
on Riemann surfaces due to Ball and Vinnikov in \cite{BV, Vin}.
As we found these not easy to read,
we were particularly pleased to be able to verify Theorem \ref{hvthm}
with our experiments.

The Helton-Vinnikov formula (\ref{amazingformula}) remains valid
when $X$ is a singular curve. In that case the period matrix, the differentials, and the
Abel-Jacobi map are meant to be defined on the
desingularization of $X$, a compact Riemann surface of genus $g$ with
$g < \frac{1}{2}(d-1)(d-2)$.  The formula holds as stated, but one no longer
 obtains all equivalence classes of symmetric determinantal representations.

The Riemann period matrix, the theta functions, their directional
derivatives, and the Abel-Jacobi-map can all be evaluated numerically
in recent versions of \texttt{Maple}. When computing the expressions
under the square roots, note that both the numerator and denominator are
 $1$-forms on $X$. Every holomorphic $1$-form on the curve $X$
can be written as $r \cdot du$, where $u=z/x$, $v= y/x$, and $r$ is a
rational function in $u$ and $v$. The {\tt algcurves} package in
\texttt{Maple} will compute $\omega$ in this form, so we obtain
$\omega_j = r_j(u,v) \cdot du$.
To evaluate the $1$-form $d(z/y)$, we 
set $h(u,v) = f(1,v,u)$ and use the identity
$ \, d h(u,v) = \frac{\partial h}{\partial u} du + \frac{\partial h}{\partial v} dv = 0$.
This implies
\begin{equation}
\label{oneform}
 d(z/y)\,\,=\,\,d(u/v)\,\,=\,\, \frac{1}{v} du - \frac{u}{v^2} dv \,\, = \,\,
\biggl(
\frac{1}{v} + \frac{u \frac{\partial h}{\partial u}}{v^2 \frac{\partial h}{\partial v}}\biggr) du,
\end{equation}
so that $d(z/y)(Q_j)=-\beta_jdu$. Under the square root signs in
(\ref{amazingformula}), the factor $du$ appears in the numerator
$\omega$ and also in the denominator (\ref{oneform}), and we cancel
it.  Hence the expressions under the square roots are rational
functions, namely $r(u,v) \cdot \nabla\theta[\epsilon](0)$ divided by
the expression in parentheses on the right in (\ref{oneform}), where
$r(u,v)$ is the vector of rational functions $r_j(u,v)$.

While the evaluation of theta functions is numerically stable, we
found the computation of the period matrix and the Abel-Jacobi map to
be more fragile. Computing the $d$ vectors $\phi(Q_j)$ is also by far
the most time-consuming step. Nonetheless, {\tt
  Maple} succeeded in correctly evaluating the Helton-Vinnikov formula
for a wide range of curves with $d\le 4$, and for some of degree
$d=5$. However, the off-diagonal entries in
(\ref{amazingformula}) we found in our computations were sometimes
wrong by a constant factor (independent of $j,k$), for reasons we do
not currently understand.

For a concrete example take the quartic
in Example \ref{ex:vinnikov}. Using the formula
(\ref{amazingformula}) we obtained all eight definite determinantal
representations $\det({\rm Id}_d x + By + Cz)$.
Our \texttt{Maple} code runs for a few minutes and finds all solutions
accurately with a precision of $20$ digits.
We verified this using the prime ideal in
Example~\ref{ex:vinnikov}.

The case of smooth quintics (genus $6$) is already a challenge.
With the help of Bernard Deconinck, we were able to compute a determinantal representation 
(\ref{amazingformula}) with an error of less than
$10^{-3}$ for the quintic polynomial at the end of Section~\ref{sec:PolynomialEquations}. However, the representation we obtained was not real (see
Remark \ref{rem:realdifficult}).

\smallskip

We conclude this section with the proof of the second part of
Thm.~\ref{equivcount} and discuss what are the suitable choices for the even theta
characteristic $\delta$ in Thm.~\ref{hvthm} that will lead to real and real
definite equivalence classes of representations. Note that a
representation obtained from the theorem is real definite if and only if
the matrix $C$ is real. The real non-definite equivalence
classes of representations correspond to the case when $C$ is a
non-real matrix for which there exists a matrix $U\in{\rm GL}_d(\C)$
such that $UU^T$, $UBU^T$, and $UCU^T$ have all real entries. Whether such $U$
exists for given complex symmetric matrices $B$ and $C$ is not at all obvious. An
explicit example is given in Ex.~\ref{thm:RealNonDef}.

For the proof of Thm.~\ref{equivcount}, we repeat the relevant part of
the statement:

\begin{theorem}\label{thm:numberrealreps}
  Let $f\in\R[x,y,z]_d$ be homogeneous and
  assume that the projective curve $X=\mathcal{V}_\C(f)$ is a smooth
  Helton-Vinnikov curve. The number of real equivalence classes
  of symmetric determinantal representations is generically either
\[
2^{g-1}(2^{k-1}+1)
\]
or one less, where $k=\lceil\frac{d}{2}\rceil$ is the number of
connected components of the set of real points $X(\R)$. Of these real equivalence classes,
exactly $2^g$ are definite.
\end{theorem}

\begin{proof}
  The result follows from work of Vinnikov \cite{Vin} on
  self-adjoint determinantal representations, which we
  apply here to our situation. By \cite[Prop.~2.2]{Vin}, a
  symplectic basis of $H_1(X,\Z)$ can be chosen in such a way that the
  Riemann period matrix $\Omega$
  satisfies $\Omega+\overline{\Omega}=H$, where $H$ is a $g\times g$
  block diagonal matrix of rank $r=g-k+1$ with $r/2$ blocks
  $\begin{bmatrix}0&1\\1&0 \end{bmatrix}$ in the top left corner and
  all other entries~zero.

  The linear symmetric determinantal representation obtained by
  Thm.~\ref{hvthm} from an even theta characteristic $\delta$ is
  equivalent to a real one if and only if $\delta$ is real,
  i.e.~invariant under the action of complex conjugation on the
  $g$-dimensional torus ${\rm Jac}(X)$. Since $X(\R)\neq\emptyset$, any conjugation-invariant divisor
  class on $X$ contains a real divisor (see \cite[Prop.~2.2]{GroHa}).
From such a divisor, one can construct a symmetric determinantal
  representation (see \cite{Bea} or \cite{Vin89}).
     When the symplectic basis of $H_1(X,\Z)$ is
  chosen as above, the action of complex conjugation on ${\rm Jac}(X)$
  is given by $\zeta={\bf u}+\Omega {\bf v}\mapsto
  \overline{\zeta}={\bf u}+\Omega(H{\bf u}-{\bf v})$ (see
  \cite[Prop.~2.2]{Vin}). For the even theta characteristic $\delta=
  {\bf a}+\Omega {\bf b}$, ${\bf a}$, ${\bf b}\in\{0,\frac 12\}^g$,
  the condition $\delta=\overline{\delta}$ in ${\rm
    Jac}(X)=\C^g/(\Z^g+\Omega\Z^g)$ thus becomes
\[
H{\bf a}\equiv 2{\bf b} \mod \Z^g.
\]
This happens if and only if $a_1=\cdots=a_r=0$. Counting
the possible choices of $a_{r+1},\dots,a_g$ and ${\bf b}$, we conclude
that there are exactly $2^{g-1}(2^{k-1}+1)$ even real theta
characteristics. From the first part of Thm.~\ref{equivcount}, we know
that when $d\equiv\pm 3\mod 8$, exactly one even theta characteristic
vanishes, i.e.~$\theta[\delta](0)=0$.
All other even theta
characteristics are non-vanishing and therefore correspond to
determinantal representations. Furthermore, by \cite[Thm.~6.1]{Vin},
an even theta characteristic $\delta= {\bf a}+\Omega{\bf b}$ will
correspond to a real definite equivalence class if and only if ${\bf
  a}=0$, and all such $\delta$ are always non-vanishing
\cite[Cor.~4.3]{Vin}. Thus there are exactly $2^g$ definite
representations, since ${\bf b}$ can be any element of $\{0,\frac
12\}^g$.
\end{proof}

\begin{example}\label{ex:8and12}
When $g=3$ and a homology basis has been picked as above, the even real
theta characteristics are given by the $12$ binary labels
\begin{equation}\label{eq:realthetachar}
\begin{small}
    \begin{bmatrix}000\\000\end{bmatrix} \begin{bmatrix}000\\001\end{bmatrix} \begin{bmatrix}000\\010\end{bmatrix} \begin{bmatrix}000\\011\end{bmatrix} \begin{bmatrix}000\\100\end{bmatrix}  \begin{bmatrix}000\\101\end{bmatrix}  \begin{bmatrix}000\\110\end{bmatrix}  \begin{bmatrix}000\\111\end{bmatrix} \;\; \begin{bmatrix}001\\000\end{bmatrix}  \begin{bmatrix}001\\010\end{bmatrix}  \begin{bmatrix}001\\100\end{bmatrix}
     \begin{bmatrix}001\\110\end{bmatrix}.
\end{small}
\end{equation}
The first eight labels correspond to the definite
representations and the last four correspond to the non-definite real
equivalence classes of representations.
\end{example}

\begin{remark}\label{rem:realdifficult}
  The characterization of real and real definite even theta
  characteristics provided by the proof of
  Thm.~\ref{thm:numberrealreps} depends on the choice of a particular
  symplectic homology basis. Unfortunately, the current \texttt{Maple}
  code for computing the period matrix does not give the user
  any control over the homology basis. This makes it hard to find real
  representations using Thm.~\ref{hvthm} in any systematic way.
\end{remark}

\section{Quartic Curves Revisited}\label{sec:quartics}

In this section we focus on the case of smooth quartic curves, 
studied in detail in \cite{PSV}, so we now fix $d=4$ and $g= 3$. 
Quartic curves are special because they have contact lines, i.e. bitangents, and 
we can explicitly write down higher degree contact curves as products of bitangents. 
This was exploited in \cite[\S 2]{PSV}, where we used azygetic triples of
bitangents as our input  to Dixon's algorithm (Section 3).

Plane quartics are canonical embeddings of genus $3$ curves \cite{Dol},
and there is a close connection between contact curves and theta functions. 
The $28$ bitangents of the curve are in bijection with the $28$ odd theta characteristics 
 $\epsilon=  {\bf a}+\Omega {\bf b}$, and this will be made explicit in  (\ref{eq:thetabitangents}) below.
The $36$ azygetic systems of contact cubics correspond to the $36$ even theta characteristics.
As seen in Example~\ref{ex:8and12}, of the resulting $36$ determinantal representations, $12$ 
are real, but only $8$ are definite. We can also derive the number $12$ from the combinatorics of the bitangents as in \cite{PSV}. 

\begin{proposition} A smooth Helton-Vinnikov quartic $\V(f)$ has exactly $12$ inequivalent representations
$f = \det(Ax+By+Cz)$ with $A,B,C$ symmetric and~real. 
\end{proposition}
 
 \begin{proof} This is a special case of Theorem~\ref{equivcount},
 however, we here give an alternative proof using the setup of \cite{PSV}.
   Let $M$ be a symmetric linear determinantal representation of $f$ and $\M$ the system of contact cubics 
 $\{u^T{\rm adj}(M)u\}\subset \C[x,y,z]_3$.  The representation $M$ is equivalent to its conjugate $\overline{M}$ if and only if the system $\M$ is real, i.e. invariant under conjugation.  The representation $M$ is equivalent to a real matrix if and only if $\M$ contains a real cubic. 
The matrix $M$ induces a labeling of the $28 = \binom{8}{2}$ bitangents, 
 $b_{ij}$, with $1 \leq i<j \leq 8$. The system $\M$ is real if and only if  conjugation acts on the bitangents via this labeling, that is, there exists $\pi \in S_8$ such that $\overline{b_{ij}} = b_{\pi(i)\pi(j)}$. Since $f$ is a Helton-Vinnikov polynomial, this permutation will be the product of four disjoint transpositions (see \cite[Table~1]{PSV}). 
 
 Suppose $\M$ is real, with permutation $\pi \in S_8$. The other $35$
   representations  (\ref{lmr1}) correspond to the $\binom{8}{4}/2$ partitions of $\{1,\hdots, 8\}$ into 
two sets of size 4. If $I | I^c$  is such a partition then the corresponding
 system of contact cubics contains $56$ products of three bitangents, namely
 $b_{ij}b_{ik}b_{i\ell}$ and $b_{im}b_{jm}b_{k\ell}$ where $i,j,k,l,m$ are distinct and $\{i,j,k,l\}=I\text{ or }I^c$.
  This system is real if and only if $\pi$ fixes the partition $I|I^c$.
 There are exactly $11$ such partitions: if $\pi = (12)(34)(56)(78)$, they~are
\begin{equation}\label{eq:realbifid}\begin{matrix}1234|5678,& 1256|3478, &1278|3456,& 1357|2468 ,&1358|2467, &1368|2457\\
1367|2458, &1457|2368, &1458|2367, & 1467|2358,& \text{ and } &1468|2357. \end{matrix}\end{equation}
Together with the system $\mathcal{M}$,  there are $12$ real systems of azygetic contact cubics. 

Next, we will show that each of these systems actually contains a real cubic. 
To do this, we use contact conics, as the product of a bitangent  with
a contact conic is a contact cubic. 
By \cite[Lemma 6.7]{PSV}, there exists a real bitangent $b \in \R[x,y,z]_1$ and a real system of contact conics $\mathcal{Q} \subset \C[x,y,z]_2$ such that their product
$\{b\cdot q\;:\; q \in \mathcal{Q}\}$ lies in the system $\M \subset \C[x,y,z]_3$. Furthermore, by 
\cite[Prop.~6.6]{PSV}, since $\V_{\R}(f)$ is nonempty, every real system of contact conics $\mathcal{Q}$ to $f$ contains a real conic $q$. The desired real contact cubic is the product $b\cdot q$.
\end{proof}

By constructing a suitable Cayley octad (see \cite[\S 3]{PSV}), the technique in the last paragraph of the above proof led us to the following result:
{\em 
There exists a smooth Helton-Vinnikov quartic $f \in \Q[x,y,z]_4$ 
that has $12$  inequivalent determinantal representations {\rm (\ref{lmr1})}
over the field $\Q$ of rational numbers.}

\begin{example}\label{thm:RealNonDef}
The special rational Helton-Vinnikov quartic we found is
\begin{align*} f(x,y,z) \;\; =& \;\;
93081 x^4 +53516  x^3 y -73684 x^2 y^2 +-31504 x y^3 +9216 y^4 \\ 
&-369150 x^2 z^2-159700 x y z^2 + 57600 y^2 z^2+90000 z^4.
\end{align*}
This polynomial satisfies
 $\,f(x,y,z) \,= {\rm det}(M)$ where
$$M\;=\begin{small}\;\begin{bmatrix}
50 x & -25 x &  - 26 x - 34 y-25 z & 9 x + 6 y + 15 z  \\
-25 x & 25 x & 27 x+ 18 y  -20 z  & - 9 x -6 y \\
- 26 x - 34 y-25 z  &   27 x  + 18 y-20 z& 108 x+72 y   &- 18 x -12 y  \\
 9 x + 6 y + 15 z & - 9 x -6 y &  - 18 x-12 y &  6 x+4 y 
\end{bmatrix}.\end{small}
$$
This representation is definite because the matrix $M$ is positive definite at the point $(1:0:0)$.
Hence $\mathcal{V}(f)$ is  a Helton-Vinnikov curve with this point in its
inner convex oval.  
Rational representatives for the other seven definite classes
are found at our website (\ref{eq:url}), along with representatives
for the four non-definite real classes. One of them is the matrix
\begin{equation}
\label{trichotomyfalse}
M_{1468} \; = \;
\begin{bmatrix}
25 x & 0 & - 32 x +12 y & -60 z \\
0 & 25 x & 10 z &  24 x+16 y  \\
- 32 x+ 12 y  & 10 z &  6 x+4 y  & 0 \\
-60 z &  24 x+16 y & 0 &  6 x+4 y 
                    \end{bmatrix}.
\end{equation}
We have ${\rm det}(M_{1468}) = 4 \cdot f(x,y,z)$, and this matrix is 
neither positive definite nor negative definite for any real values of $x,y,z $.
Any equivalent representation of a multiple of $f$ in the form
$\det({\rm Id}_4x+By+Cz)$ considered in Sections~\ref{sec:PolynomialEquations}~and~\ref{sec:theta}
cannot have all entries of $C$ real.  One such representation, for a suitable $U\in {\rm GL}_4(\C)$, is
\[U^TM_{1468}U \;=\; \begin{small}\begin{bmatrix}
x +\frac{64}{71} y& 0 & -\frac{23}{1349} \sqrt{26980} \;i\;z& \frac{-51}{1633} \sqrt{16330}\;z  \\
0 & x+ \frac{2}{3}y & -\frac{2}{19} \sqrt{570}\;z & \frac{4}{23} \sqrt{345}\;i\;z \\
-\frac{23}{1349}  \sqrt{26980} \;i\;z & -\frac{2}{19} \, \sqrt{570}\;z & x-\frac{4}{19} y & 0 \\
\frac{-51}{1633} \,\sqrt{16330} \;z & \frac{4}{23} \sqrt{345}\;i\;z  & 0 & x-\frac{18}{23}y
\end{bmatrix}\end{small}.\] 
\end{example}

\smallskip

The correspondence between bitangents and odd theta characteristics 
can be understood abstractly via the isomorphism
between the Jacobian of $X$ and the divisor class group ${\rm Cl}^0(X)$,   
 and  can be turned into an explicit formula for the
bitangents. Let $u=z/x$, $v=y/x$ and write $h(u,v)=f(1,v,u)$. Then a basis for the
$3$-dimensional complex vector space of
holomorphic $1$-forms on $X$ is given by
\[
(\omega_1,\omega_2,\omega_3)=\left(\frac{du}{\partial h/\partial v},
  \frac{vdu}{\partial h/\partial v},
  \frac{udu}{\partial h/\partial v}\right).
\]
Let $(\Omega_1|\Omega_2)$ be the period matrix of $f$ with respect to
$(\omega_1,\omega_2,\omega_3)$ and any symplectic basis of
$H_1(X,\Z)$. Then given an odd theta characteristic $\epsilon={\bf
  a}+\Omega{\bf b}$, the corresponding bitangent is defined by the
linear form
\begin{equation}\label{eq:thetabitangents}
b_\epsilon(x,y,z) =
\bigl(\nabla\theta[\epsilon](0)\bigr)^T\cdot\Omega_1^{-1}\cdot (x,y,z)^T,
\end{equation}
where $\nabla\theta[\epsilon]$ is the gradient of $\theta[\epsilon]$
in the three complex variables $z_1,z_2,z_3$. For the proof,
see Dolgachev \cite[Section 5.5.4]{Dol}. This holds independently of
the symplectic basis of $H_1(X,\Z)$, but a change of that
basis will permute the bitangents.

The formula (\ref{eq:thetabitangents}) can be evaluated using the
{\tt Maple} code described in Section~4. This allows us to compute the
$8 {\times} 8$-bitangent matrix $(b_{ij})$ of \cite[Eq.~3.4]{PSV}
directly from the Riemann period matrix $\Omega$ of the curve $X$, using a technique due to Riemann described in 
\cite[\S 2]{Gua}. In this manner, one computes the
symmetric determinantal representations (\ref{lmr1})
of the curve $X$ directly from the period matrix $\Omega$.
This computation seems to be a key ingredient in constructing
explicit three-phase solutions of the Kadomtsev-Petviashvili equation \cite{DFS},
and we hope  that the combinatorial tools developed here and in \cite{PSV}
will be useful for  integrable systems.

One of the earliest papers on algorithms for theta functions in genus three
was written by Arthur Cayley in 1897.
In \cite{Cay}  he gives a concrete bijection between 
the bitangents $b_{ij}$ of a plane quartic and the
odd theta characteristics, and also between the 
classes $I | I^c$ of determinantal representations and the even theta characteristics.
We here reproduce a relabeled version of the table in Cayley's article:
\smallskip
\[ \begin{tabular}{c||c|c|c|c|c|c|c|c}
$ 2{\bf b}\backslash2{\bf a} $& 000 & 100 & 010 & 110 & 001 & 101 & 011 &111\\
\hline\hline
000 & Ê$\emptyset$ & 1238 & 1267 & 1245 & \textbf{1468} & 1578 & 1356 & 1347\\
100 & \textbf{1234} & ÊÊ48 & 1258 & ÊÊ35 & \textbf{1457} & ÊÊ16 & 1378 & ÊÊ27\\
010 & \textbf{1256} & 1247 & ÊÊ57 & Ê46 & \textbf{1367} & 1345 & ÊÊ23 & ÊÊ18\\
110 & \textbf{1278} & ÊÊ37 & ÊÊ68 & 1236 & \textbf{1358} & ÊÊ25 & ÊÊ14 & 1567\\
001 & \textbf{1357} & 1346 & 1478 & 1568 & ÊÊ12 & ÊÊ38 & ÊÊ67 & ÊÊ45\\
101 & \textbf{1368} & ÊÊ26 & 1456 & ÊÊ17 & ÊÊ34 & 1248 & ÊÊ58 & 1235\\
011 & \textbf{1458} & 1678 & ÊÊ13 & ÊÊ28 & ÊÊ56 & ÊÊ47 & 1257 & 1246\\
111 & \textbf{1467} & ÊÊ15 & ÊÊ24 & 1348 & ÊÊ78 & 1237 & 1268 & ÊÊ36
\end{tabular}\]
\smallskip
Here a partition $I|I^c$ of $\{1,\hdots, 8\}$ is represented by the $4$-tuple $I$ which contains
the index $1$.  For instance, the $4$-tuple $1238$ corresponds to the even theta characteristic
$ \begin{bmatrix}100\\000\end{bmatrix} $. Each partition $I|I^c$ represents a Cremona transformation leading
to a new representation (\ref{lmr1}) as described in
\cite[\S 3]{PSV}. The twelve
$4$-tuples marked in bold face are the real equivalence classes, and
this gives a bijection between the lists 
in (\ref{eq:realthetachar}) and in (\ref{eq:realbifid}). 
Likewise, the pairs $ij$ in Cayley's table represent bitangents $b_{ij}$
and the corresponding odd theta characteristics. For instance, the odd characteristic
$ \begin{bmatrix}100\\111\end{bmatrix} $ represents the bitangent
$b_{15}$. In this manner, we can parametrize the 28
bitangents of all plane quartics explicitly
with odd theta functions.

\smallskip

Experts in moduli of curves will be quick to point out that 
this parametrization should extend from smooth curves to all
stable curves. This is indeed the case. For instance, four distinct lines
form a stable Helton-Vinnikov quartic such~as
$$ f(x,y,z) \quad = \quad x y z (x+y+z). $$
\begin{figure}
 \includegraphics[width=6.0cm]{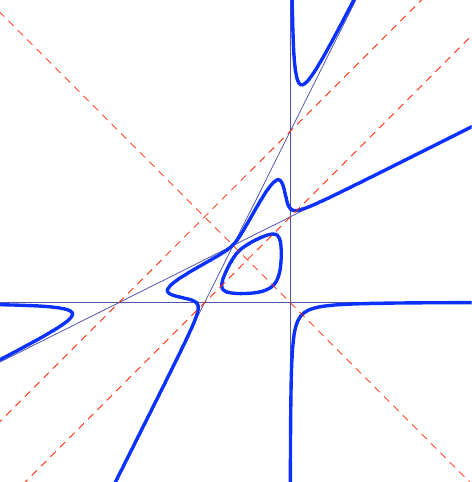} \quad
\vskip -0.2cm
\caption{Degeneration of a Helton-Vinnikov quartic into four lines.}
\label{fig:4lines}
\end{figure}
The bitangent matrix $(b_{ij})$ of this reducible curve has $7$ distinct non-zero entries:
$$
\begin{small}
\begin{bmatrix}
               0 & z & y & y + z & x & x + z & x + y & x {+} y {+} z \\
               z & 0 & y + z & y & x + z & x & x {+} y {+} z & x + y \\
               y & y + z & 0 & z & x + y & x {+} y {+} z & x & x + z \\
               y + z & y & z & 0 & x {+} y {+} z & x + y & x + z & x \\
               x & x + z & x + y & x {+} y {+} z & 0 & z & y & y + z \\
               x + z & x & x {+} y {+} z & x + y & z & 0 & y + z & y \\
               x + y & x {+} y {+} z & x & x + z & y & y + z & 0 & z \\
               x {+} y {+} z & x + y & x + z & x & y + z & y & z & 0 \\
\end{bmatrix}.
\end{small}
$$
All principal $4 {\times} 4$-minors of this $8 {\times} 8$-matrix
are multiples of $f(x,y,z)$, most of them non-zero. They are
all in the same equivalence class, which is real but not definite.
The entries in the bitangent matrix indicate a partition of
the $28$ odd theta characteristics into seven groups of four.
For instance, the antidiagonal entry $x+y+z$
corresponds to the four entries $18, 27, 36$ and $45$ in Cayley's table,
and hence to the four odd theta characteristics
$\, \begin{bmatrix}010\\111\end{bmatrix} $,
$\, \begin{bmatrix}100\\111\end{bmatrix} $,
$\, \begin{bmatrix}111\\111\end{bmatrix} $ and
$\, \begin{bmatrix}001\\111\end{bmatrix} $.

If we consider a family of smooth quartics that degenerates to 
the reducible quartic $f(x,y,z)$, then its bitangent matrix will degenerate
to the above $8 {\times} 8$-matrix, and hence the $28$ distinct
bitangents of the smooth curve bunch up in seven clusters of four.
This degeneration is visualized in Figure~\ref{fig:4lines}.
Among the seven limit bitangents are the three lines spanned by
pairs of intersection points.

Algebraically, such a degenerating family can be described 
as a curve over a field with a valuation, such as the
field of {\em real Puiseux series} $\R\{\!\{ \epsilon \}\!\}$.
The notions of spectrahedra
and Helton-Vinnikov curves makes perfect sense over the real closed field
$\R\{\!\{ \epsilon \}\!\}$. This has been investigated from the perspective
of tropical geometry by David Speyer, who proved in \cite{Spe} that 
tropicalized Helton-Vinnikov curves are precisely {\em honeycomb curves}.
We believe that the tropicalization in \cite{Spe} offers yet another approach to constructing
linear determinantal representations (\ref{lmr1}), in addition to the three methods presented here,
and we hope to return to this topic.

\bigskip 

\noindent
 {\bf Acknowledgments.} We wish to thank
  Charles Chen,   Bernard Deconinck,
  Didier Henrion, Chris Swierczewski and 
  Victor Vinnikov for discussions and computational contributions
  that were very helpful to us in the preparation of this article.

{\linespread{1}

}


\begin{thebibliography}{10}
\providecommand{\url}[1]{\texttt{#1}}
\providecommand{\urlprefix}{URL }

\bibitem{AD}
M.~Artebani and I.~Dolgachev:
The Hesse pencil of a plane cubic curve.
{\em Enseign.~Math.} \textbf{55}, 235--273,~2009.


\bibitem{BV}
J.~A. Ball and V.~Vinnikov:
\newblock Zero-pole interpolation for matrix meromorphic functions on a compact
  {R}iemann surface and a matrix {F}ay trisecant identity.
\newblock \emph{Amer. J. Math.}, \textbf{121~(4)}, 841--888, 1999.

\bibitem{Bea}
A.~Beauville:
\newblock Determinantal hypersurfaces.
\newblock \emph{Michigan Math. Journal}, \textbf{48}, 39--64, 2000.

\bibitem{bertini}
D.~Bates, J. Hauenstein, A. Sommese, and C.  Wampler:
{\sc Bertini}: Software for Numerical Algebraic Geometry,
{\tt http://www.nd.edu/$\sim$sommese/bertini/}, (2010).

\bibitem{Cay}
A.~Cayley: Algorithm for the characteristics of the triple
  $\vartheta$-functions.
  {\em   Journal f\"ur die reine und angewandte Mathematik},
   {\bf 87}, 165-169, 1879.

\bibitem{singular}
W.~Decker, G.-M.~Greuel, G.~Pfister, and H.~Sch{\"o}nemann:
{\sc Singular}: A computer algebra system for polynomial computations,
 {\tt www.singular.uni-kl.de} (2010).

\bibitem{DeEtAl}
B.~Deconinck, M.~Heil, A.~Bobenko, M.~van Hoeij, and M.~Schmies:
\newblock Computing {R}iemann theta functions.
\newblock \emph{Mathematics of Computation}, \textbf{73~(247)}, 1417--1442, 2004.

\bibitem{DP}
B.~Deconinck and M.~S. Patterson:
\newblock Computing the {A}bel map.
\newblock \emph{Physica D}, \textbf{237~(24)}, 3214--3232, 2008.

\bibitem{DvH}
B.~Deconinck and M.~van Hoeij:
\newblock Computing {R}iemann matrices of algebraic curves.
\newblock \emph{Physica D}, \textbf{152/153}, 28--46, 2001.

\bibitem{Dix}
A.C.~Dixon:
\newblock {Note on the reduction of a ternary quantic to a symmetrical
  determinant}.
\newblock \emph{Cambr. Proc.} \textbf{11}, 350--351, 1902.

\bibitem{Dol}
I.~Dolgachev:
\newblock \emph{Classical Algebraic Geometry: A Modern View},
\newblock Cambridge University Press, 2012.

\bibitem{DFS}
B.~Dubrovin, R.~Flickinger and H.~Segur:
Three-phase solutions of the Kadomtsev-Petviashvili equation.
{\em Stud. Appl. Math.} {\bf 99} (1997) 137--203.

\bibitem{Gre}
W.~Greub:
\newblock \emph{Linear Algebra}.
\newblock Springer-Verlag, New York, 4th edn., 1975.
\newblock Graduate Texts in Math., No 23.



\bibitem{GroHa}
B.~H. Gross and J.~Harris:
\newblock Real algebraic curves.
\newblock \emph{Ann. Sci. \'Ecole Norm. Sup. (4)}, \textbf{14~(2)}, 157--182,
  1981.

\bibitem{Gua}
J.~Guardia: On the Torelli problem and Jacobian Nullwerte in genus three,Ê
Ê \emph{Michigan Mathematical Journal}, \textbf{60}, 51--65, 2011.

\bibitem{HV}
J.~W. Helton and V.~Vinnikov:
\newblock Linear matrix inequality representation of sets.
\newblock \emph{Comm. Pure Appl. Math.}, \textbf{60~(5)}, 654--674, 2007.

\bibitem{Hul} K.~Hulek: {\em Elementary Algebraic Geometry},
Student Mathematical Library, Vol.~20, American Mathematical Society, Providence, RI, 2003.

\bibitem{LPR} A.~Lewis, P.~Parrilo and M.~Ramana: The Lax conjecture is true.
\emph{Proceedings Amer.~Math.~Soc.}, \textbf{133}, 2495--2499, 2005.

\bibitem{MB}
T.~Meyer-Brandis.
\newblock Ber{\"u}hrungssysteme und symmetrische Darstellungen ebener Kurven,
  1998.
\newblock Diplomarbeit, Universit{\"a}t Mainz, 
written under the supervision of
 D.~van Straten, posted at
\url{http://enriques.mathematik.uni-mainz.de/straten/diploms}

\bibitem{Mum}
D.~Mumford:
\newblock \emph{Tata Lectures on Theta. {I}}.
\newblock Modern Birkh\"auser Classics. Birkh\"auser, Boston, MA,
  2007.   Reprint of the 1983 edition.

\bibitem{PSV}
D.~Plaumann, B.~Sturmfels, and C.~Vinzant:
\newblock Quartic curves and their bitangents, Ê \emph{Journal of Symbolic Computation}, \textbf{46}, 712-733, 2011.Ê
%
%
\bibitem{Spe} D.~Speyer:
Horn's problem, Vinnikov curves, and the hive cone.
{\em Duke Math. J.} {\bf 127} (2005), no. 3, 395--427.

\bibitem{Vin89}
V.~Vinnikov:
\newblock Complete description of determinantal representations of smooth
  irreducible curves.
\newblock \emph{Linear Algebra Appl.}, \textbf{125}, 103--140, 1989.

\bibitem{Vin}
V.~Vinnikov:
\newblock Selfadjoint determinantal representations of real plane curves.
\newblock \emph{Mathematische Annalen}, \textbf{296~(3)}, 453--479, 1993.

\bigskip

\end{thebibliography}
\end{document}